\documentclass[a4paper,10pt,leqno]{article}

\RequirePackage[T1]{fontenc}
\RequirePackage[utf8]{inputenc}
\RequirePackage{xcolor}
	\definecolor{sapphire}{HTML}{0F52BA}
\RequirePackage{amsfonts,amssymb,amsmath}
\RequirePackage{amsthm}
\RequirePackage{comment,lmodern,mathtools,stmaryrd,microtype,tikz-cd}

\RequirePackage{hyperref}
	\hypersetup{colorlinks,allcolors=sapphire}
\usepackage[capitalize]{cleveref}
	
\numberwithin{equation}{section}
\theoremstyle{plain}
	\newtheorem{thm}{Theorem}
	
	\newtheorem{cor}[thm]{Corollary}
	
	\newtheorem{prop}[thm]{Proposition}
\theoremstyle{definition}
	\newtheorem{ex}[thm]{Example}
\theoremstyle{remark}	
	\newtheorem{rem}[thm]{Remark}

\usepackage{autonum}

\newcommand{\bb}{\mathbf}

\newcommand{\fr}{\mathfrak}
\renewcommand{\sf}{\mathsf}

\newcommand{\sapphire}[1]{\textcolor{sapphire}{#1}}
	\newcommand{\dfemph}[1]{\sapphire{\emph{#1}}}

\DeclareMathOperator{\tr}{tr}

\newcommand{\GL}{\mathrm{GL}}
	\newcommand{\SU}{\mathrm{SU}}
	\newcommand{\SO}{\mathrm{SO}}
	\newcommand{\UL}{\mathrm{U}}

\newcommand{\mat}[1]{\begin{matrix}#1\end{matrix}}
\newcommand{\pmat}[1]{\left(\mat{#1}\right)}

\RequirePackage{authblk}

\newcommand{\mytitle}{Free Immersions in Critical Dimension from Lie Groups}
\title{\texorpdfstring{\vspace{-0.5in}\large\bfseries\mytitle}{\mytitle}}
\author{Minh-T\^{a}m Quang Trinh}
\affil{Department of Mathematics \\ Howard University\footnote{\url{minhtam.trinh@howard.edu}}}
\date{September 22, 2026}

\begin{document}

\maketitle

\begin{abstract}
We construct free immersions of $P^3 \times S^1$ and $S^3 \times S^1$ into $14$-space.
These are the first examples of free immersions in critical dimension of closed manifolds that are not spheres, projective spaces, tori, or surfaces.
The idea is to mimic De Leo's construction for $m$-tori with $m \leq 5$.
De Leo constructs a map that is symmetric with respect to a representation of the Lie group formed by the $(m - 1)$-torus; we replace this group with $\mathrm{SO}(3)$.
In our examples, the osculating determinant is constant.
\end{abstract}

\thispagestyle{empty}



\setcounter{section}{-1}
\section{Introduction}

\subsection{Background}

Fix integers $m, q > 0$ and a smooth manifold $M$ of dimension $m$.
Recall that an \dfemph{immersion} of $M$ into $\bb{R}^q$ is a $C^1$ map $f = (f^1, \ldots, f^q) \colon M \to \bb{R}^q$ for which the Jacobian matrix $(\partial_\alpha f^i)_{\alpha, i}$ has maximal rank $m$ everywhere.
Following~\cite[\S{1.1.4}]{gromov}, we say that $f$ is \dfemph{free} if, more strongly, it is $C^2$ and the osculating matrix 
\begin{align} 
(\partial_\alpha f^i, \partial_{\alpha\beta} f^i)_{(\alpha, i), (\alpha\beta, i)}
\end{align} 
has maximal rank 
\begin{align}
q_m \vcentcolon= m + \binom{m + 1}{2} = \frac{m(m + 3)}{2}
\end{align}
everywhere.
Since this forces $q \geq q_m$, we say that $q_m$ is the \dfemph{critical dimension} for free immersions of $m$-manifolds.

As defined above, freeness is not an intrinsic property:
The rank of the osculating matrix is sensitive to changes of local coordinates on $\bb{R}^q$.
However, it is preserved by any affine transformation of $\bb{R}^q$.
More generally, there is an intrinsic notion of freeness for immersions into any smooth manifold equipped with an affine connection.

Beyond~\cite{gromov}, one motivation for the study of free immersions is their appearance in Nash's work on embeddings of Riemannian manifolds.
See the metric perturbation theorem in~\cite{nash}. 

Eliashberg and Gromov showed that if either $q > q_m$ or $M$ is open, then the partial differential relation of being a free immersion of $M$ into $\bb{R}^q$ satisfies the $h$-principle~\cite{eg}.
So in this case, such free immersions exist whenever $M$ is parallelizable.
This focuses our attention on the existence of free immersions when $q = q_m$ and $M$ is closed and connected.

For the real projective $m$-space $P^m$, there is a straightforward construction of a free immersion into $\bb{R}^{q_m}$, by showing that the Veronese embedding $P^m \to P^{q_m}$ factors through an embedded $\bb{R}^{q_m}$.
Thus, there is also a free immersion of the $m$-sphere $S^m$ into $\bb{R}^{q_m}$ via the quotient map $S^m \to P^m$.

Perhaps the next simplest example of a closed connected manifold that exists in every dimension $m$ is the $m$-torus $(S^1)^m$.
In the breakthrough work~\cite{deleo}, De Leo constructs an explicit free immersion
\begin{align} 
F_m \colon (S^1)^m \to \bb{R}^{q_m}
	\quad\text{for $m = 2, 3, 4, 5$}.
\end{align} 
From the cases of $S^2$, $P^2$, and $(S^1)^2$, he deduces a free immersion of any closed surface into $\bb{R}^{q_2} = \bb{R}^5$ by a gluing argument. 

The key idea in~\cite{deleo} is to choose $F_m$ to be symmetric with respect to a representation of the (real) Lie group $\UL(1)^{m - 1}$ on $\bb{R}^q$.
More precisely, once we write the torus coordinates as 
\begin{align} 
(\vec{x}, z) \vcentcolon= (x_1, \ldots, x_{m - 1}, z) \in (S^1)^{m - 1} \times S^1 = (S^1)^m,
\end{align} 
and identify $S^1$ with $\UL(1)$, De Leo's $F_m$ takes the form
\begin{align}
F_m(\vec{x}, z) = R_m(\vec{x}) \cdot v_m(z)
\end{align}
for some homomorphism $R_m \colon \UL(1)^{m - 1} \to \SO(q_m)$ and map $v_m \colon S^1 \to \bb{R}^{q_m}$.
He then reduces the problem to choosing $v_m$.

After we posted the first version of this note to arXiv, we learned from De Leo that he had posted~\cite{deleo_2} a few days earlier.
This sequel to~\cite{deleo} presents a construction whose input consists of free immersions $(S^1)^a \to \bb{R}^{q_a}$ and $(S^1)^b \to \bb{R}^{q_b}$ along with a ``cross-free'' map $(S^1)^{a + b} \to \bb{R}^{ab}$, whose output is a free immersion $(S^1)^{a + b} \to \bb{R}^{q_{a + b}}$.
As a corollary, De Leo obtains an explicit free immersion $(S^1)^m \to \bb{R}^{q_m}$ for every dimension $m$.

\subsection{Results}

The purpose of this note is to show that the method of~\cite{deleo} applies beyond tori.
In our view, the algebraic structure that makes the homomorphism $R_m$ useful is not just symmetry, but the simple behavior of $R_m$ under differentiation.
This behavior, in turn, is an instance of a general fact about how representations of Lie groups transform under invariant vector fields:
See \eqref{eq:derivative} in \S\ref{subsec:derivatives}.

The smallest compact Lie groups beyond $\UL(1)$ are the rotation group $\SO(3)$ and its universal cover, the special unitary group $\SU(2)$.
There is a close analogy between $\UL(1)$, viewed as the unit circle in the complex plane, and $\SU(2)$, viewed as the unit $3$-sphere in the $4$-space of quaternions.
This leads us to hope for analogues of the free immersions $F_m$ in which $\UL(1)^{m - 1}$ is replaced by $\SO(3)$ or $\SU(2)$.
As these groups are respectively diffeomorphic to $P^3$ and $S^3$, we are proposing to construct free immersions of $P^3 \times S^1$ and $S^3 \times S^1$ into $\bb{R}^{q_4} = \bb{R}^{14}$.
These are our main results.

\begin{thm}\label{thm:main}
There is a free immersion $F \colon P^3 \times S^1 \to \bb{R}^{14}$ of the form
\begin{align}
F(g, z) = R(g) \cdot v(z),
\end{align}
for some homomorphism $R \colon \SO(3) \to \GL_{14}(\bb{R})$ and smooth map $v \colon S^1 \to \bb{R}^{14}$, once we identify $P^3$ with $\SO(3)$.

In fact, there is a five-parameter family of such maps $F$.
For any fixed choice of parameters, the determinant of the osculating matrix is constant in $(g, z) \in \SO(3) \times S^1$.
\end{thm}

The constancy of the osculating determinant is somewhat remarkable.
De Leo has pointed out to us that in general, a free map of the form $F = R \cdot v$ can always be made to have constant osculating determinant by a reparametrization of $v$.
However, this usually makes the formula for $v$ more complicated, whereas our $v$ turns out to be very clean.

\begin{cor}\label{cor:main}
The analogue of \Cref{thm:main} holds with $S^3$ and $\SU(2)$ in place of $P^3$ and $\SO(3)$.
\end{cor}

\begin{proof}
Via the quotient map $\SU(2) \to \SO(3)$.
\end{proof}

Let us mention that the $4$-manifolds $P^3 \times S^1$ and $S^3 \times S^1$ are of particular interest in complex geometry.
A \dfemph{primary Hopf surface} is a quotient of $\bb{C}^2 - \{0\}$ by an infinite cyclic group generated by some complex-analytic contraction.
More generally, a \dfemph{Hopf surface} is a free quotient of a primary Hopf surface by a finite group.
Such complex surfaces form the simplest examples of non-K\"ahler compact complex manifolds.
In~\cite{kodaira}, Kodaira proved that a complex surface is a primary Hopf surface if and only if it is diffeomorphic to $S^3 \times S^1$.
As a corollary, other Hopf surfaces are diffeomorphic to $P^3 \times S^1$.

In light of \Cref{thm:main} and \Cref{cor:main}, it is natural to ask whether $P^2 \times S^1$ and $S^2 \times S^1$ also admit free immerions in critical dimension.
Recall that a choice of axis lets us identify $S^2$ with $\SO(2)\backslash\SO(3)$.
This suggests that one could bootstrap the free immersion of $S^2 \times S^1$ from \Cref{thm:main}, by forming the quotients of appropriate group actions.

It also seems possible that, by extending De Leo's newer work~\cite{deleo_2}, one could bootstrap free immersions in critical dimension of manifolds of the form $P^3 \times (S^1)^m$ and $S^3 \times (S^1)^m$ from \Cref{thm:main} and \Cref{cor:main}.

\subsection{Acknowledgments}

I thank Roberto De Leo for his kind interest.
I thank GPT 5.6 Sol for symbolic-algebra computations that helped me simplify the Ansatz for $v$ in \Cref{sec:p-3-times-s-1}.
No part of this text was generated by AI.

\section{Lie Groups}

This section reviews background material on Lie groups that is not strictly needed for~\cite{deleo}, but needed for our work.
It may be found in many places in the literature, as the representation theory of the Lie groups most relevant to us, $\UL(1)$ and $\SU(2)$ and $\SO(3)$, is fundamental to the physics of hydrogen-like atoms, though we do not discuss this further.
\emph{Henceforth, all vector spaces will be assumed finite-dimensional}.

\subsection{Coordinates for Groups}\label{subsec:coords}

Given any vector space $V$, we write $\GL(V)$ for its group of linear automorphisms and $\fr{gl}(V)$ for the Lie algebra of $\GL(V)$.
If $V = \bb{F}^n$, then we write $\GL_n(\bb{F})$ and $\fr{gl}_n(\bb{F})$ instead.
We write $I_n \in \GL_n(\bb{F})$ for the identity matrix.
We will also use the standard notations $\UL(n)$, $\SU(n)$, $\SO(n)$ and $\fr{sl}_n(\bb{F})$, $\fr{su}(n)$, $\fr{so}(n)$.

Famously, $\SU(2)$ and $\SO(3)$ are related by a degree-$2$ covering homomorphism $\SU(2) \to \SO(3)$.
It will be useful to describe the embedding $\fr{su}(2) \subseteq \fr{sl}_2(\bb{C})$ and the isomorphism $\fr{su}(2) \to \fr{so}(3)$ in explicit coordinates.
Recall that
\begin{align}
\fr{sl}_n(\bb{F}) &= \{\xi \in \fr{gl}_n(\bb{F}) \mid \tr \xi = 0\},\\
\fr{su}(n) &= \{\xi \in \fr{sl}_n(\bb{C}) \mid \xi^\ast + \xi = 0\},\\
\fr{so}(n)& = \{\xi \in \fr{sl}_n(\bb{R}) \mid \xi^\ast + \xi = 0\}.
\end{align}
Take the standard basis for $\fr{sl}_2(\bb{F})$:
\begin{align}
	e \vcentcolon= \pmat{0&1\\ 0&0},\quad 
	f \vcentcolon= \pmat{0&0\\ 1&0},\quad 
	h \vcentcolon= \pmat{1&0\\ 0&-1}.
\end{align}
Then the following elements form a basis for $\fr{su}(2)$:
\begin{align}
\begin{array}{r@{\:}ll}
\Xi_1 &\vcentcolon= \pmat{0&-\frac{\sqrt{-1}}{2}\\ -\frac{\sqrt{-1}}{2}&0}
		&= -\tfrac{\sqrt{-1}}{2}(e + f),\\[3ex]
\Xi_2 &\vcentcolon= \pmat{0&-\frac{1}{2}\\ \frac{1}{2}&0}
		&= -\tfrac{1}{2}(e - f),\\[3ex]
\Xi_3 &\vcentcolon= \pmat{-\frac{\sqrt{-1}}{2}&0\\ 0&\frac{\sqrt{-1}}{2}}
		&= -\tfrac{\sqrt{-1}}{2}h.
	\end{array}
\end{align}
(That is, $\Xi_i = -\frac{\sqrt{-1}}{2}\sigma_i$ for $i = 1,2,3$, where $\sigma_1, \sigma_2, \sigma_3$ are the Pauli matrices.)
Meanwhile, we can define $\fr{so}(3)$ in a coordinate-free way as the Lie algebra freely generated over $\bb{R}$ by elements $\xi_1, \xi_2, \xi_3$ subject to the relations
\begin{align} 
{[\xi_1, \xi_2]} &= \xi_3,\\
{[\xi_2, \xi_3]} &= \xi_1,\\
{[\xi_3, \xi_1]} &= \xi_2.
\end{align} 
By checking the analogous relations among $\Xi_1, \Xi_2, \Xi_3$, we deduce an isomorphism of Lie algebras $\fr{su}(2) \xrightarrow{\sim} \fr{so}(3)$ sending $\Xi_i \mapsto \xi_i$ for $i = 1, 2, 3$.

Since $\SU(2)$ is simply connected, this isomorphism is the differential $d\varpi$ of some homomorphism of Lie groups 
\begin{align} 
	\varpi \colon \SU(2) \to \SO(3).
\end{align} 
It turns out that $\varpi$ descends to an isomorphism $\SU(2)/\langle \pm I_2\rangle \xrightarrow{\sim} \SO(3)$.

\subsection{Representations}\label{subsec:reps}

Recall that a \dfemph{representation} of a Lie group $G$ on a vector space $V$ is an action of $G$ through a homomorphism into $\GL(V)$.
An isomorphism of representations is a $G$-equivariant linear isomorphism.
If $G$ is compact, then the representation necessarily decomposes as a direct sum of \emph{irreducible} representations, or \dfemph{irreps}, up to isomorphism.

\begin{ex}[The circle]\label{ex:r-mod-z}
Let $C = \bb{R}/2\pi \bb{Z}$.
For each integer $w$, there is a complex irrep $\psi_w \colon C \to \GL_1(\bb{C})$ given by 
\begin{align} 
\psi_w(x) \vcentcolon= \exp (\sqrt{-1}wx),
\end{align} 
called the irrep of \dfemph{winding number} or \dfemph{weight} $w$.
Note that $\psi_1$ identifies $C$ with $\UL(1)$.
Up to isomorphism, these are the only complex irreps.

\end{ex} 

\begin{ex}[The torus]\label{ex:torus}
Every complex irrep of the $d$-fold power $C^d$ is a $d$-fold tensor product of complex irreps of $C$.
So by \Cref{ex:r-mod-z}, the complex irreps are again $1$-dimensional, but now classified by integer $d$-tuples $\vec{w} \in \bb{Z}^d$, still called their \dfemph{weights}.
Explicitly, the complex irrep $\psi_{\vec{w}} \colon C^d \to \GL_1(\bb{C})$ is given by
\begin{align}
\psi_{\vec{w}}(\vec{x}) = \exp(\sqrt{-1}\vec{w} \cdot \vec{x}),
\end{align}
where $- \cdot -$ denotes the dot product.

For a complex representation of $C^d$ to be the complexification of a real one, it must decompose into irreps so that the resulting multiset of weights is stable under $\vec{w} \mapsto -\vec{w}$.
In particular, $\psi_{\vec{0}}$ complexifies the trivial real irrep, which we will denote by $\varphi_{\vec{0}}$, while for each weight $\vec{w} \neq \vec{0}$, the sum $\psi_{\vec{w}} \oplus \psi_{-\vec{w}}$ complexifies a real irrep $\varphi_{\vec{w}} \colon C^d \to \GL_2(\bb{R})$.
Explicitly,
\begin{align} 
\varphi_{\vec{w}}(\vec{x}) \vcentcolon=
	\pmat{
		\cos \vec{w} \cdot \vec{x}
		& -{\sin \vec{w} \cdot \vec{x}}\\ 
		\sin \vec{w} \cdot \vec{x}
		& \cos \vec{w} \cdot \vec{x}
	}.
\end{align} 
Note that $\varphi_{\vec{w}}$ and $\varphi_{-\vec{w}}$ are conjugate, hence isomorphic, for any $\vec{w}$.
Up to isomorphism, the $\varphi_{\vec{w}}$ for $\vec{w} \in \bb{Z}^d/\langle \pm 1\rangle$ are the only real irreps.
Finally, note that in the $d = 1$ case, $\varphi_1$ identifies $C$ with $\SO(2)$.
\end{ex}

\begin{ex}[$\mathrm{SU}(2)$]\label{ex:su-2}
For each integer $k \geq 0$, there is a complex irrep \begin{align} 
\sigma_k \colon \SU(2) \to \GL(W_k),
\end{align} 
where $W_k$ is the $(k + 1)$-dimensional vector space of homogeneous polynomial functions on $\bb{C}^2$ of degree $k$.
Here, $\SU(2)$ acts on $W_k$ by pullback of functions along its action on $\bb{C}^2$.
We say that $\sigma_k$ has \dfemph{weight} $k$.
Up to isomorphism, these are the only complex irreps.

For $\sigma_k$ to be the complexification of a real irrep, $k$ must be even.
In this case, $-I_2 \in \SU(2)$ acts trivially, so $\sigma_k$ factors through the quotient map $\varpi \colon \SU(2) \to \SO(3)$ from \Cref{subsec:coords}.
Up to isomorphism, the real irreps of $\SU(2)$ arising in this way are the only real irreps.
\end{ex}

\begin{ex}[$\mathrm{SO}(3)$]\label{ex:so-3}
Every complex irrep of $\SO(3)$ arises from some complex irrep of $\SU(2)$ of even weight.
That is, in the notation of \Cref{ex:su-2}, every complex irrep of $\SO(3)$ takes the form $\bar{\sigma}_k \colon \SO(3) \to \GL(W_k)$ for some even $k \geq 0$, in which case it is characterized by $\sigma_k = \bar{\sigma}_k \circ \varpi$.

Write $k = 2\ell$.
Then $\sigma_k$ complexifies an irrep of $\SU(2)$ on some $(2\ell + 1)$-dimensional real vector space $V_\ell$.
Therefore, $\bar{\sigma}_k$ complexifies an irrep of $\SO(3)$ of the form
\begin{align} 
\rho_\ell \colon \SO(3) \to \GL(V_\ell).
\end{align} 
Explicitly, one may take $V_\ell$ to be the vector space of homogeneous, degree-$\ell$ harmonic polynomial functions on $\bb{R}^3$.
(The restrictions of these functions to $S^2$ are known as \dfemph{spherical harmonics}.)
Up to isomorphism, the $\rho_\ell$ are the only real irreps of $\SO(3)$.
\end{ex}

\subsection{Directional Derivatives}\label{subsec:derivatives}

Recall that for a general Lie group $G$ with Lie algebra $\fr{g}$, every element $\xi \in \fr{g}$ defines left- and right-invariant vector fields on $G$ by translation.
At $g \in G$, the left-invariant vector field is given by
\begin{align}
g \cdot \xi = \!\left.\tfrac{d}{dt} (g \cdot \exp(t\xi)) \right|_{t \to 0},
\end{align}
where $\exp \colon \fr{g} \to G$ is the usual exponential map.
More generally, if $G$ acts smoothly from the right on a smooth manifold $M$, then $\xi$ defines a vector field on $M$, whose value at $p \in M$ is given by
\begin{align}
p \cdot \xi \vcentcolon= \!\left.\tfrac{d}{dt} (p \cdot \exp(t\xi)) \right|_{t \to 0}.
\end{align}
In particular, for any smooth function $f \colon M \to N$ with differential $df \colon TM \to TN$, we have a directional derivative
\begin{align}
\partial_\xi f(p) 
	&\vcentcolon= df( p \cdot \xi )
	= \!\left.\tfrac{d}{dt} f(p \cdot \exp(t\xi)) \right|_{t \to 0}.
\end{align}
For a representation $\rho \colon G \to \GL(V)$, this formula becomes
\begin{align}\label{eq:derivative}
\partial_\xi \rho(g)
	= \rho(g) \cdot d\rho(\xi),
\end{align}
where $d\rho(\xi) = \!\left.\frac{d}{dt} \rho(\exp(t\xi)) \right|_{t \to 0}$.

\subsection{Coordinates for Representations}\label{subsec:rep-coords}

For the generators $\xi_1, \xi_2, \xi_3 \in \fr{so}(3)$ discussed in \Cref{subsec:coords}, and the real irreps $\rho_\ell \colon \SO(3) \to \GL(V_\ell)$ discussed in \Cref{ex:so-3}, we now derive explicit formulas for the differentials $d\rho_\ell(\xi_i) \in \fr{gl}(V_\ell)$ in coordinates on $V_\ell$ where they become especially clean.
Again, these formulas should exist in the literature, but we find it pleasant to provide a more concise account.

By our prior discussion, it suffices to compute
the differentials $d\sigma_{2\ell}(\Xi_i) \in \fr{gl}(W_{2\ell})$, then find good coordinates for an $\fr{su}(2)$-stable real structure on $W_{2\ell}$.
Let $\sf{x}, \sf{y}$ be the standard coordinates on $\bb{C}^2$, so that $W_{2\ell}$ is the degree-$2\ell$ part of $\bb{C}[\sf{x}, \sf{y}]$.
Let $\{\sf{w}_j^\ell\}_{-\ell \leq j \leq \ell}$
be the monomial basis of $W_{2\ell}$ in which 
\begin{align} 
\sf{w}_j^\ell \vcentcolon= \binom{2\ell}{\ell - j} \sf{x}^{\ell + j} \sf{y}^{\ell - j}.
\end{align} 
For convenience, let $\sf{w}_{-(\ell + 1)}^\ell, \sf{w}_{\ell + 1}^\ell \vcentcolon= 0$.

The $\SU(2)$-action on $W_{2\ell}$ is the restriction of the $\GL_2(\bb{C})$-action
\begin{align}
\pmat{a&b\\ c&d} \cdot \sf{x}^m \sf{y}^n = (a\sf{x} + c\sf{y})^m (b\sf{x} + d\sf{y})^n.
\end{align}
Via the differential of the latter, the elements $e, f, h \in \fr{sl}_2(\bb{C})$ act by
\begin{align}
\left.\begin{array}{@{}r@{\:}l}
e \cdot \sf{w}_j^\ell &= (\ell + j + 1) \sf{w}_{j + 1}^\ell,\\[1ex]
f \cdot \sf{w}_j^\ell &= (\ell - j + 1)\sf{w}_{j - 1}^\ell,\\[1ex]
h \cdot \sf{w}_j^\ell &= 2j\sf{w}_j^\ell
\end{array}\right\} 
\quad\text{for $-\ell \leq j \leq \ell$}.
\end{align}
Expanding the $\Xi_i$ in terms of $e, f, h$, we obtain
\begin{align}
\left.\begin{array}{@{}r@{\:}l}
\Xi_1 \cdot \sf{w}_j^\ell &= -\frac{\sqrt{-1}}{2} ((\ell + j + 1) \sf{w}_{j + 1}^\ell + (\ell - j + 1)\sf{w}_{j - 1}^\ell),\\[1ex]
\Xi_2 \cdot \sf{w}_j^\ell &= -\frac{1}{2} ((\ell + j + 1) \sf{w}_{j + 1}^\ell - (\ell - j + 1)\sf{w}_{j - 1}^\ell),\\[1ex]
\Xi_3 \cdot \sf{w}_j^\ell &= -\sqrt{-1} j\sf{w}_j^\ell
\end{array}\right\} 
\quad\text{for $-\ell \leq j \leq \ell$}.
\end{align}
In analogy with the formation of cosine and sine from $e^{ix}$ and $e^{-ix}$, let
\begin{align}
\left.\begin{array}{@{}r@{\:}l}
\sf{u}_j^\ell &\vcentcolon= \frac{\sqrt{-1}}{2} (\sf{w}_j - (-1)^j\sf{w}_{-j}),\\[1ex]
\sf{v}_j^\ell &\vcentcolon= \frac{1}{2} (\sf{w}_j + (-1)^j\sf{w}_{-j})
\end{array}\right\}
\quad\text{for $-\ell \leq j \leq \ell$}.
\end{align}
Note that $\sf{u}_0^\ell = 0$ and $\sf{v}_0^\ell = \sf{w}_0$.
For convenience, let $\sf{u}_{\ell + 1}^\ell, \sf{v}_{\ell + 1}^\ell \vcentcolon= 0$.
Then, in the basis of $W_{2\ell}$ formed by $\sf{v}_0^\ell, \sf{u}_1^\ell, \sf{v}_1^\ell, \ldots, \sf{u}_\ell^\ell, \sf{v}_\ell^\ell$, we have
\begin{align}
\Xi_1 \cdot \sf{v}_0^\ell &= -(\ell + 1)\sf{u}_1^\ell,\\
	\Xi_1 \cdot \sf{u}_j^\ell &= \tfrac{1}{2}((\ell + j + 1)\sf{v}_{j + 1}^\ell + (\ell - j + 1)\sf{v}_{j - 1}^\ell),\\
	\Xi_1 \cdot \sf{v}_j^\ell &= -\tfrac{1}{2}((\ell + j + 1)\sf{u}_{j + 1}^\ell + (\ell - j + 1)\sf{u}_{j - 1}^\ell),\\
\Xi_2 \cdot \sf{v}_0^\ell &= -(\ell + 1)\sf{v}_1^\ell,\\
	\Xi_2 \cdot \sf{u}_j^\ell &= -\tfrac{1}{2}((\ell + j + 1)\sf{u}_{j + 1}^\ell - (\ell - j + 1)\sf{u}_{j - 1}^\ell),\\
	\Xi_2 \cdot \sf{v}_j^\ell &= -\tfrac{1}{2}((\ell + j + 1)\sf{v}_{j + 1}^\ell - (\ell - j + 1)\sf{v}_{j - 1}^\ell),\\
\Xi_3 \cdot \sf{v}_0^\ell &= 0,\\
	\Xi_3 \cdot \sf{u}_j^\ell &= j\sf{v}_j^\ell,\\
	\Xi_3 \cdot \sf{v}_j^\ell &= -j\sf{u}_j^\ell.
\end{align}
So the real span of this basis forms an $\fr{su}(2)$-stable real subspace $V_\ell \subseteq W_{2\ell}$.
Moreover, the formulas above completely describe the matrices for $d\rho_\ell(\xi_i) = d\sigma_{2\ell}(\Xi_i) \in \fr{gl}(V_\ell)$ with respect to the ordered basis.

\begin{ex}\label{ex:d-rho-xi}
We list $d\rho_\ell(\xi_1), d\rho_\ell(\xi_2), d\rho_\ell(\xi_3)$, in that order, for small $\ell$.
\begin{enumerate} 
\item[$\ell = 1$:]
\begin{align}
\pmat{\cdot&\frac{1}{2}&\cdot\\ 
	-2&\cdot&\cdot\\ 
	\cdot&\cdot&\cdot},\quad 
\pmat{\cdot&\cdot&\frac{1}{2}\\ 
	\cdot&\cdot&\cdot\\ 
	-2&\cdot&\cdot},\quad 
\pmat{\cdot&\cdot&\cdot\\ 
	\cdot&\cdot&-1\\ 
	\cdot&1&\cdot}.
\end{align}

\item[$\ell = 2$:]
\begin{align}
&\pmat{\cdot&1&\cdot&\cdot&\cdot\\
	-3&\cdot&\cdot&\cdot&-\frac{1}{2}\\
	\cdot&\cdot&\cdot&\frac{1}{2}&\cdot\\ 
	\cdot&\cdot&-2&\cdot&\cdot\\ 
	\cdot&2&\cdot&\cdot&\cdot},\quad 
\pmat{\cdot&\cdot&1&\cdot&\cdot\\
	\cdot&\cdot&\cdot&\frac{1}{2}&\cdot\\
	-3&\cdot&\cdot&\cdot&\frac{1}{2}\\ 
	\cdot&-2&\cdot&\cdot&\cdot\\ 
	\cdot&\cdot&-2&\cdot&\cdot},\\[0.5ex]
&\pmat{\cdot&\cdot&\cdot&\cdot&\cdot\\
	\cdot&\cdot&-1&\cdot&\cdot\\
	\cdot&1&\cdot&\cdot&\cdot\\
	\cdot&\cdot&\cdot&\cdot&-2\\
	\cdot&\cdot&\cdot&2&\cdot}.
\end{align}

\item[$\ell = 3$:]
\begin{align}
&\small\pmat{\cdot&\frac{3}{2}&\cdot&\cdot&\cdot&\cdot&\cdot\\
	-4&\cdot&\cdot&\cdot&-1&\cdot&\cdot\\ 
	\cdot&\cdot&\cdot&1&\cdot&\cdot&\cdot\\ \cdot&\cdot&-\frac{5}{2}&\cdot&\cdot&\cdot&-\frac{1}{2}\\ \cdot&\frac{5}{2}&\cdot&\cdot&\cdot&\frac{1}{2}&\cdot\\ \cdot&\cdot&\cdot&\cdot&-3&\cdot&\cdot\\ \cdot&\cdot&\cdot&3&\cdot&\cdot&\cdot},\quad 
\pmat{\cdot&\cdot&\frac{3}{2}&\cdot&\cdot&\cdot&\cdot\\
	\cdot&\cdot&\cdot&1&\cdot&\cdot&\cdot\\
	-4&\cdot&\cdot&\cdot&1&\cdot&\cdot\\ \cdot&-\frac{5}{2}&\cdot&\cdot&\cdot&\frac{1}{2}&\cdot\\ \cdot&\cdot&-\frac{5}{2}&\cdot&\cdot&\cdot&\frac{1}{2}\\ \cdot&\cdot&\cdot&-3&\cdot&\cdot&\cdot\\
	\cdot&\cdot&\cdot&\cdot&-3&\cdot&\cdot},\\[0.5ex]
&\small\pmat{\cdot&\cdot&\cdot&\cdot&\cdot&\cdot&\cdot\\
	\cdot&\cdot&-1&\cdot&\cdot&\cdot&\cdot\\ \cdot&1&\cdot&\cdot&\cdot&\cdot&\cdot\\ \cdot&\cdot&\cdot&\cdot&-2&\cdot&\cdot\\ \cdot&\cdot&\cdot&2&\cdot&\cdot&\cdot\\ \cdot&\cdot&\cdot&\cdot&\cdot&\cdot&-3\\ \cdot&\cdot&\cdot&\cdot&\cdot&3&\cdot}.
\end{align}

\end{enumerate}
\mbox{}
\end{ex}

\section{Freeness Criteria}

Fix a real Lie group $G$ of dimension $d$, a finite-dimensional real vector space $V$, a representation $R \colon G \to \GL(V)$, and a $C^2$ map $v \colon S^1 \to V$.
Let
\begin{align} 
F \colon G \times S^1 \to V
\end{align}
be defined by 
\begin{align}\label{eq:factorization}
F(g, z) = R(g) \cdot v(z).
\end{align}
We now explain how \eqref{eq:factorization} simplifies the study of the freeness of $F$.

Let $\fr{g}$ be the Lie algebra of $G$.
Choose a basis $(\xi_\alpha)_\alpha$ for $\fr{g}$, and set
\begin{align}
X_\alpha = dR(\xi_\alpha)
\quad\text{for all $\alpha$}.
\end{align}
By \eqref{eq:derivative}, the first- and second-order partials of $R$ with respect to the left-invariant vector fields induced by the $\xi_\alpha$ are, respectively,
\begin{align}
\begin{array}{l@{\quad}l}
R X_\alpha
	&\text{for $1 \leq \alpha \leq d$},\\
R X_\alpha X_\beta
	&\text{for $1 \leq \alpha \leq \beta \leq d$}.
\end{array}
\end{align}
Similarly, the first-order partials of $F$ with respect to these vector fields are
\begin{align}
\begin{array}{l@{\quad}l}
R X_\alpha \cdot v
	&\text{for $1 \leq \alpha \leq d$},\\
R \cdot v',
\end{array}
\end{align}
while the second-order partials of $F$ are
\begin{align}
\begin{array}{l@{\quad}l}
R X_\alpha X_\beta \cdot v
	&\text{for $1 \leq \alpha \leq \beta \leq d$},\\
R X_\alpha \cdot v'
	&\text{for $1 \leq \alpha \leq d$},\\
R \cdot v''.
\end{array}
\end{align}
Since the values of $R$ are invertible linear operators, we deduce:

\begin{prop}\label{prop:partials}
\begin{enumerate} 
\item 	$G \xrightarrow{R} \GL(V) \to \fr{gl}(V)$ is a free immersion if and only if the list of matrices
\begin{align}\label{eq:list-r}
	X_\alpha,\quad
	X_\alpha X_\beta
\end{align} 
is linearly independent.

\item 	$F$ is a free immersion if and only if the list of vectors 
\begin{align}\label{eq:list-f}
v',\quad 
v'',\quad 
X_\alpha \cdot v,\quad
X_\alpha \cdot v',\quad 
X_\alpha X_\beta \cdot v
\end{align} 
is linearly independent for all $z \in S^1$.

\end{enumerate}
\end{prop} 

\begin{cor}\label{cor:partials}
If $F$ is a free immersion, then $G \xrightarrow{R} \GL(V) \to \fr{gl}(V)$ is a free immersion.
\end{cor}

\begin{proof}
Immediate from comparing \eqref{eq:list-r} and \eqref{eq:list-f}.
\end{proof}

\begin{cor}\label{cor:dim-2}
If $G$ is compact, $F$ is a free immersion, and $\dim V = q_{d + 1}$, then the $G$-invariant subspace $V^G \subseteq V$ has dimension at most $2$.
\end{cor}

\begin{proof}
Since $G$ is compact, we can pick an invariant projection $V \to V^G$.
The operators $X_\alpha$ act by zero on $V^G$, so in the list of vectors \eqref{eq:list-f}, only $v'$ and $v''$ can possibly have nonzero projections to $V^G$.
But if $\dim V = q_{d + 1}$ and \eqref{eq:list-f} is linearly independent, then \eqref{eq:list-f} actually forms a basis for $V$.
Then the projections of $v'$ and $v''$ to $V^G$ must span $V^G$.
\end{proof}

\begin{ex}\label{ex:quadratic}
In the notation of \Cref{ex:torus}, take $G = C^d$ and $V = \bb{R}^q$ and 
\begin{align} 
R = \varphi_{\vec{w}^{(1)}} \oplus \cdots \oplus \varphi_{\vec{w}^{(r)}} \oplus \varphi_{\vec{0}}^{\oplus s},
\end{align} 
for some $r, s$ such that $q = 2r + s$ and some weights $\vec{w}^{(i)} = (w_1^{(i)}, \ldots, w_d^{(i)}) \in \bb{Z}^d$.
Taking $\xi_\alpha = \partial_{x_\alpha}$ for $1 \leq \alpha \leq d$, we compute
\begin{align}\label{eq:de-leo-x}
X_\alpha = w^{(1)}_\alpha J \oplus \cdots \oplus w^{(r)}_\alpha J \oplus 0_s,
	\quad\text{where $J = \pmat{0 & -1 \\ 1 & 0}$},
\end{align}
from which
\begin{align}
X_\alpha X_\beta = -w_\alpha^{(1)}w_\beta^{(1)} I_2 \oplus \cdots \oplus -w_\alpha^{(r)}w_\beta^{(r)} I_2 \oplus 0_s.
\end{align} 
(Above, $0_s$ is the $s \times s$ zero matrix.)

We deduce that for the $X_\alpha X_\beta$ to be linearly independent, it is necessary that whenever a (real) quadratic $k$-form $Q(\vec{x}) = \sum_{\alpha, \beta} c_{\alpha\beta} x_\alpha x_\beta$ satisfies $Q(\vec{w}^{(1)}) = \cdots = Q(\vec{w}^{(r)}) = 0$, we must have $c_{\alpha\beta}= 0$ for all $\alpha, \beta$.
In other words, for $C^d \xrightarrow{R} \GL(V) \to \fr{gl}(V)$ to be free, the evaluation map from the vector space of quadratic $d$-forms into $\bb{R}^r$ that sends
\begin{align}\label{eq:quadratic}
	Q \mapsto (Q(\vec{w}^{(1)}), \ldots, Q(\vec{w}^{(d)}))
\end{align}
must be injective.
In particular, we must have 
\begin{align} 
\binom{d + 1}{2} \leq r \leq \lfloor q/2\rfloor.
\end{align}
When $q = q_{d + 1}$, this inequality forces $d \leq 4$.

By \Cref{cor:partials}, the discussion above refines the discussion in~\cite[\S{2.5}]{deleo}, which showed directly that a map $F \colon (S^1)^m \to \bb{R}^{q_m}$ of the form \eqref{eq:factorization}, via $C \simeq S^1$, can be free only if $m \leq 5$.
\end{ex}

\begin{rem} 
In the setup of \Cref{ex:quadratic}, the freeness of $C^d \xrightarrow{R} \GL(V) \to \fr{gl}(V)$ also forces the weights $\vec{w}^{(r)}, \ldots, \vec{w}^{(r)}$ to be linearly independent as vectors in $\bb{R}^d$.
For otherwise, they all vanish on some nonzero linear functional $\lambda$ on $\bb{R}^d$.
But then the symmetric square of $\lambda$ is a nonzero quadratic form that vanishes on each weight.

Note that the linear independence of the weights is a tighter constraint than the linear independence of the $X_\alpha$, due to \eqref{eq:de-leo-x} and the bound $\binom{d + 1}{2} \leq r$.
\end{rem} 

\begin{ex}\label{ex:veronese-so-3}
Take $G = \SO(3)$.
In the notation of \Cref{ex:so-3}, take $V = V_1$ and $R = \rho_1$.
Here, one can check that the nine matrices $X_\alpha, X_\alpha X_\beta$ are indeed linearly independent, so $\rho_1$ gives rise to a free immersion 
\begin{align}
\SO(3) \to \fr{gl}(V_1) \simeq \bb{R}^9.
\end{align}
Since $q_3 = 9$, this is a free immersion in critical dimension.

In fact, after identifying $\SO(3)$ with $P^3$, one can check that this map is just a rescaled version of the free immersion of $P^3$ into $\bb{R}^9$ that arises from the Veronese embedding of $P^3$ into $P^9$, discussed in the introduction.
\end{ex}

\section{\texorpdfstring{The $P^3 \times S^1$ Construction}{The P³ × S¹ Construction}}\label{sec:p-3-times-s-1}

Henceforth, we focus on maps $F \colon \SO(3) \times S^1 \to \bb{R}^{14}$ of the form \eqref{eq:factorization}.

\subsection{$R$}\label{subsec:r}

The possibilities for $R$ are constrained.
In the notation of \Cref{ex:so-3}, $R$ must be isomorphic to a direct sum of the irreps $\rho_\ell$ for integers $\ell \geq 0$, where the underlying vector space of $\rho_\ell$ has dimension $2\ell + 1$.
In particular, these dimensions must sum to $14$.
Moreover, by \Cref{cor:dim-2}, the multiplicity of the trivial irrep $\rho_0$ must be $\leq 2$.
So the possibilities are:
\begin{align} 
\begin{array}{ll}
\text{decomposition of $R$ into irreps}
	&\text{multiset of irrep dimensions}\\
	\hline 
\rho_0^2 \oplus \rho_1 \oplus \rho_4 
	&1,1,3,9\\
\rho_0^2 \oplus \rho_2 \oplus \rho_3 
	&1,1,5,7\\
\rho_0 \oplus \rho_1^2 \oplus \rho_3
	&1,3,3,7\\
\rho_0 \oplus \rho_1 \oplus \rho_2^2
	&1,3,5,5\\
\rho_1^3 \oplus \rho_2
	&3,3,3,5
\end{array}
\end{align}
It is preferable to maximize the multiplicity of $\rho_0$, since this makes the determinant needed to check the linear independence in \Cref{prop:partials}(2) easier to compute.
At the same time, it is preferable to keep the irrep dimensions small, to keep the operators $X_\alpha$ and $X_\alpha X_\beta$ in \Cref{prop:partials} manageable.
We choose
\begin{align}\label{eq:5-7}
R = \rho_2 \oplus \rho_3 \oplus \rho_0^2,
\end{align} 
now putting $\rho_0$ at the end of the sum for visual consistency with \Cref{ex:quadratic} and~\cite{deleo}.
Henceforth, we regard the codomain of $F$ as $V_2 \oplus V_3 \oplus V_0^2$.
We choose the operators $X_\alpha$ to take the form
\begin{align} 
X_\alpha = d\rho_2(\xi_\alpha) \oplus d\rho_3(\xi_\alpha) \oplus 0_2
\quad\text{for $\alpha = 1,2,3$},
\end{align} 
where $d\rho_2(\xi_\alpha)$ and $d\rho_3(\xi_\alpha)$ are written out explicitly in \Cref{ex:d-rho-xi}.

To compute the products $X_\alpha X_\beta$, we simply compute $X_{\alpha\beta}^{(\ell)} \vcentcolon= d\rho_\ell(\xi_\alpha) d\rho_\ell(\xi_\beta)$ for $\ell = 2, 3$.
In the order 
\begin{align} 
X_{11}^{(\ell)},\quad
X_{12}^{(\ell)},\quad
X_{13}^{(\ell)},\quad
X_{22}^{(\ell)},\quad
X_{23}^{(\ell)},\quad
X_{33}^{(\ell)},
\end{align} 
they are as follows:
\begin{enumerate} 
\item[$\ell = 2$:]
\begin{align}
&\pmat{-3&\cdot&\cdot&\cdot&-\frac12\\
		\cdot&-4&\cdot&\cdot&\cdot\\
		\cdot&\cdot&-1&\cdot&\cdot\\
		\cdot&\cdot&\cdot&-1&\cdot\\
		-6&\cdot&\cdot&\cdot&-1},\quad 
	\pmat{\cdot&\cdot&\cdot&\frac12&\cdot\\
		\cdot&\cdot&-2&\cdot&\cdot\\
		\cdot&-1&\cdot&\cdot&\cdot\\
		6&\cdot&\cdot&\cdot&-1\\
		\cdot&\cdot&\cdot&1&\cdot},\\[0.5ex]
&\pmat{\cdot&\cdot&-1&\cdot&\cdot\\
		\cdot&\cdot&\cdot&-1&\cdot\\
		\cdot&\cdot&\cdot&\cdot&-1\\
		\cdot&-2&\cdot&\cdot&\cdot\\
		\cdot&\cdot&-2&\cdot&\cdot},\quad
	\pmat{-3&\cdot&\cdot&\cdot&\frac12\\
		\cdot&-1&\cdot&\cdot&\cdot\\
		\cdot&\cdot&-4&\cdot&\cdot\\
		\cdot&\cdot&\cdot&-1&\cdot\\
		6&\cdot&\cdot&\cdot&-1},\\[0.5ex]
&\pmat{\cdot&1&\cdot&\cdot&\cdot\\
		\cdot&\cdot&\cdot&\cdot&-1\\
		\cdot&\cdot&\cdot&1&\cdot\\
		\cdot&\cdot&2&\cdot&\cdot\\
		\cdot&-2&\cdot&\cdot&\cdot},\quad
	\pmat{\cdot&\cdot&\cdot&\cdot&\cdot\\
		\cdot&-1&\cdot&\cdot&\cdot\\
		\cdot&\cdot&-1&\cdot&\cdot\\
		\cdot&\cdot&\cdot&-4&\cdot\\
		\cdot&\cdot&\cdot&\cdot&-4}.
\end{align}

\item[$\ell = 3$:]
\begin{align}
&\small\pmat{-6&\cdot&\cdot&\cdot&-\frac32&\cdot&\cdot\\
		\cdot&-\frac{17}{2}&\cdot&\cdot&\cdot&-\frac12&\cdot\\
		\cdot&\cdot&-\frac52&\cdot&\cdot&\cdot&-\frac12\\
		\cdot&\cdot&\cdot&-4&\cdot&\cdot&\cdot\\
		-10&\cdot&\cdot&\cdot&-4&\cdot&\cdot\\
		\cdot&-\frac{15}{2}&\cdot&\cdot&\cdot&-\frac32&\cdot\\
		\cdot&\cdot&-\frac{15}{2}&\cdot&\cdot&\cdot&-\frac32},\quad 
	\pmat{\cdot&\cdot&\cdot&\frac32&\cdot&\cdot&\cdot\\
		\cdot&\cdot&-\frac72&\cdot&\cdot&\cdot&-\frac12\\
		\cdot&-\frac52&\cdot&\cdot&\cdot&\frac12&\cdot\\
		10&\cdot&\cdot&\cdot&-1&\cdot&\cdot\\
		\cdot&\cdot&\cdot&1&\cdot&\cdot&\cdot\\
		\cdot&\cdot&\frac{15}{2}&\cdot&\cdot&\cdot&-\frac32\\
		\cdot&-\frac{15}{2}&\cdot&\cdot&\cdot&\frac32&\cdot},\\[0.5ex]
&\small\pmat{\cdot&\cdot&-\frac32&\cdot&\cdot&\cdot&\cdot\\
		\cdot&\cdot&\cdot&-2&\cdot&\cdot&\cdot\\
		\cdot&\cdot&\cdot&\cdot&-2&\cdot&\cdot\\
		\cdot&-\frac52&\cdot&\cdot&\cdot&-\frac32&\cdot\\
		\cdot&\cdot&-\frac52&\cdot&\cdot&\cdot&-\frac32\\
		\cdot&\cdot&\cdot&-6&\cdot&\cdot&\cdot\\
		\cdot&\cdot&\cdot&\cdot&-6&\cdot&\cdot},\quad
	\pmat{-6&\cdot&\cdot&\cdot&\frac32&\cdot&\cdot\\
		\cdot&-\frac52&\cdot&\cdot&\cdot&\frac12&\cdot\\
		\cdot&\cdot&-\frac{17}{2}&\cdot&\cdot&\cdot&\frac12\\
		\cdot&\cdot&\cdot&-4&\cdot&\cdot&\cdot\\
		10&\cdot&\cdot&\cdot&-4&\cdot&\cdot\\
		\cdot&\frac{15}{2}&\cdot&\cdot&\cdot&-\frac32&\cdot\\
		\cdot&\cdot&\frac{15}{2}&\cdot&\cdot&\cdot&-\frac32},\\[0.5ex]
&\small\pmat{\cdot&\frac32&\cdot&\cdot&\cdot&\cdot&\cdot\\
		\cdot&\cdot&\cdot&\cdot&-2&\cdot&\cdot\\
		\cdot&\cdot&\cdot&2&\cdot&\cdot&\cdot\\
		\cdot&\cdot&\frac52&\cdot&\cdot&\cdot&-\frac32\\
		\cdot&-\frac52&\cdot&\cdot&\cdot&\frac32&\cdot\\
		\cdot&\cdot&\cdot&\cdot&6&\cdot&\cdot\\
		\cdot&\cdot&\cdot&-6&\cdot&\cdot&\cdot},\quad
	\pmat{\cdot&\cdot&\cdot&\cdot&\cdot&\cdot&\cdot\\
		\cdot&-1&\cdot&\cdot&\cdot&\cdot&\cdot\\
		\cdot&\cdot&-1&\cdot&\cdot&\cdot&\cdot\\
		\cdot&\cdot&\cdot&-4&\cdot&\cdot&\cdot\\
		\cdot&\cdot&\cdot&\cdot&-4&\cdot&\cdot\\
		\cdot&\cdot&\cdot&\cdot&\cdot&-9&\cdot\\
		\cdot&\cdot&\cdot&\cdot&\cdot&\cdot&-9}.
\end{align}

\end{enumerate}
\mbox{}

\subsection{$v$}

In De Leo's free immersions of $(S^1)^2$ and $(S^1)^3$~\cite[\S{2.1}--{2.2}]{deleo}, the Ansatz for $v \colon S^1 \to \bb{R}^q$ is simply the parametrization of an ellipse:
\begin{align}\label{eq:ellipse}
v(z) = A + (\cos z) B + (\sin z) C
\quad\text{for some $A, B, C \in \bb{R}^q$}.
\end{align}
In his free immersions of $(S^1)^4$ and $(S^1)^5$~\cite[\S{2.3}--{2.4}]{deleo}, it takes this form outside of one coordinate.

Moreover, the ellipse above is constructed so that its projection into most irreducible summands of $R \colon \UL(1)^{m - 1} \to \bb{R}^q$ remain ellipses.

This suggests that we try an Ansatz of the form \eqref{eq:ellipse}, whose projections onto $V_2$ and $V_3$ via \eqref{eq:5-7} remain ellipses.
We assume that
\begin{align}
v(z) = \pmat{v_2(z) \\ v_3(z) \\ \gamma \cos z \\ \gamma \sin z}
\quad\text{for some ellipses $v_\ell \colon S^1 \to V_\ell$ and $\gamma \neq 0$}.
\end{align}
We credit GPT-5.6 Sol with observing that our $v$ can have zero entries, in contrast to De Leo's $v$'s.
Ultimately, we choose $v_2$ and $v_3$ to be the \emph{circles}
\begin{align}
v_2 = \pmat{a\\ 0\\ 0\\ \alpha \cos z\\ \alpha \sin z},\quad 
v_3 = \pmat{b\\ \beta \cos z\\ \beta \sin z\\ 0\\ 0\\ 0\\ 0}
\quad\text{for some $a, b, \alpha, \beta \neq 0$}.
\end{align}
Let $c = \cos z$ and $s = \sin z$.
Then $\pmat{v' &v'' &X_\alpha \cdot v &X_\alpha \cdot v' &X_\alpha X_\beta \cdot v}$ is
\begin{align} 
\left(\begin{matrix}
\cdot&\cdot&\cdot&\cdot&\cdot&\cdot&\cdot&\cdot\\
\cdot&\cdot&-3a-\frac{\alpha s}{2}&\frac{\alpha c}{2}&\cdot&-\frac{\alpha c}{2}&-\frac{\alpha s}{2}&\cdot\\
\cdot&\cdot&\frac{\alpha c}{2}&-3a+\frac{\alpha s}{2}&\cdot&-\frac{\alpha s}{2}&\frac{\alpha c}{2}&\cdot\\
-\alpha s&-\alpha c&\cdot&\cdot&-2\alpha s&\cdot&\cdot&-2\alpha c\\
\alpha c&-\alpha s&\cdot&\cdot&2\alpha c&\cdot&\cdot&-2\alpha s\\
\cdot&\cdot&\frac{3\beta c}{2}&\frac{3\beta s}{2}&\cdot&-\frac{3\beta s}{2}&\frac{3\beta c}{2}&\cdot\\
-\beta s&-\beta c&-4b&\cdot&-\beta s&\cdot&\cdot&-\beta c\\
\beta c&-\beta s&\cdot&-4b&\beta c&\cdot&\cdot&-\beta s\\
\cdot&\cdot&-\frac{5\beta s}{2}&-\frac{5\beta c}{2}&\cdot&-\frac{5\beta c}{2}&\frac{5\beta s}{2}&\cdot\\
\cdot&\cdot&\frac{5\beta c}{2}&-\frac{5\beta s}{2}&\cdot&-\frac{5\beta s}{2}&-\frac{5\beta c}{2}&\cdot\\
\cdot&\cdot&\cdot&\cdot&\cdot&\cdot&\cdot&\cdot\\
\cdot&\cdot&\cdot&\cdot&\cdot&\cdot&\cdot&\cdot\\
-\gamma s&-\gamma c&\cdot&\cdot&\cdot&\cdot&\cdot&\cdot\\
\gamma c&-\gamma s&\cdot&\cdot&\cdot&\cdot&\cdot&\cdot
\end{matrix}\right.
\end{align} 
\begin{align} 
\left.\begin{matrix}
-3a-\frac{\alpha s}{2}&\frac{\alpha c}{2}&\cdot&-3a+\frac{\alpha s}{2}&\cdot&\cdot\\
\cdot&\cdot&-\alpha c&\cdot&-\alpha s&\cdot\\
\cdot&\cdot&-\alpha s&\cdot&\alpha c&\cdot\\
-\alpha c&6a-\alpha s&\cdot&-\alpha c&\cdot&-4\alpha c\\
-6a-\alpha s&\alpha c&\cdot&6a-\alpha s&\cdot&-4\alpha s\\
-6b&\cdot&-\frac{3\beta s}{2}&-6b&\frac{3\beta c}{2}&\cdot\\
-\frac{17\beta c}{2}&-\frac{7\beta s}{2}&\cdot&-\frac{5\beta c}{2}&\cdot&-\beta c\\
-\frac{5\beta s}{2}&-\frac{5\beta c}{2}&\cdot&-\frac{17\beta s}{2}&\cdot&-\beta s\\
\cdot&10b&-\frac{5\beta c}{2}&\cdot&\frac{5\beta s}{2}&\cdot\\
-10b&\cdot&-\frac{5\beta s}{2}&10b&-\frac{5\beta c}{2}&\cdot\\
-\frac{15\beta c}{2}&\frac{15\beta s}{2}&\cdot&\frac{15\beta c}{2}&\cdot&\cdot\\
-\frac{15\beta s}{2}&-\frac{15\beta c}{2}&\cdot&\frac{15\beta s}{2}&\cdot&\cdot\\
\cdot&\cdot&\cdot&\cdot&\cdot&\cdot\\
\cdot&\cdot&\cdot&\cdot&\cdot&\cdot
\end{matrix}\right).
\end{align} 
We arrive at
\begin{align}
\det \pmat{v' &v'' &X_\alpha \cdot v &X_\alpha \cdot v' &X_\alpha X_\beta \cdot v}
	&= -\frac{15^4}{2} ab \alpha^4 \beta^6 \gamma^2 (c^2 + s^2)^6\\[1ex]
	&= -\frac{50625}{2} ab \alpha^4 \beta^6 \gamma^2.
\end{align}
Since $a, b, \alpha, \beta, \gamma \neq 0$, this is simply a nonzero constant.
By \Cref{prop:partials}, this proves \Cref{thm:main}.

\begin{rem}
In the notation of \S\ref{subsec:rep-coords}, our choices for $v_2$ and $v_3$ are
\begin{align}
v_2(z) &= a\sf{v}_0^2 + \alpha ((\cos z) \sf{u}_2^2 + (\sin z) \sf{v}_2^2),\\
v_3(z) &= b\sf{v}_0^2 + \beta ((\cos z) \sf{u}_1^3 + (\sin z) \sf{v}_1^3).
\end{align}
There are other choices that would suffice for \Cref{thm:main}, replacing the pair of subscripts $2$ and $1$ by another pair of distinct subscripts $j_2 \in \{1,2\}$ and $j_3 \in \{1,2,3\}$.
In this generality, one checks that
\begin{align}\label{eq:x-3}
X_3 v_2 = j_2 v'_2,\qquad 
X_3 v_3 &= j_3 v'_3.
\end{align} 
These identities show why $j_2$ and $j_3$ must be distinct.
For if they are the same integer $j$, then $X_\alpha v' = j X_\alpha X_3 v$ for each $\alpha$, violating linear independence.

Let us give a conceptual explanation of \eqref{eq:x-3}.
Let $\iota \colon C \to \SO(3)$ be the homomorphism $\iota(z) = \exp(z \xi_3)$.
In the notation of \Cref{ex:torus}, we see that for any integers $\ell \geq k \geq 1$, the vectors $\sf{u}_k^\ell$ and $\sf{v}_k^\ell$ span an irreducible summand of 
\begin{align}
\rho_\ell \circ \iota \colon C \to \SO(3) \to \GL(V_\ell)
\end{align}
of weight $k$.
Therefore, $X_3 = d\rho_\ell(\xi_3)$ acts on $\langle\sf{u}_k^\ell, \sf{v}_k^\ell\rangle$ by $kJ$, where $J$ is defined as in \eqref{eq:de-leo-x}.
Meanwhile, from \S\ref{subsec:rep-coords}, we know that $X_3$ sends $\sf{v}_0^\ell$ to zero.
Taking $(\ell, k) = (2, j_2), (3, j_3)$, we arrive at \eqref{eq:x-3}.
\end{rem}

\begin{rem}
The calculation of the $14 \times 14$ osculating determinant can be simplified through several observations.
First, it immediately decomposes as the product of a $12 \times 12$ determinant and the $2 \times 2$ determinant 
\begin{align} 
\det \pmat{-\gamma s &-\gamma c\\ \gamma c &-\gamma s} = \gamma^2(c^2 + s^2) = \gamma^2.
\end{align}
The $12 \times 12$ determinant involves rows $11$ and $12$ of the original $14 \times 14$ matrix.
These rows only contain nonzero entries in columns $9$, $10$, and $12$.
So the $12 \times 12$ determinant becomes a sum of three $10 \times 10$ determinants.
Finally, the expansions of the latter involve many $2 \times 2$ determinants that simplify to scalar multiples of $c^2 + s^2 = 1$.

This does not completely explain why the osculating determinant is constant, as the original matrix includes several entries that are linear in $\sin z$ and survive into the $10\times 10$ determinants.
It is remarkable that all of these extra copies of $\sin z$ cancel in the end.
It would be pleasant to have a purely conceptual explanation for the constancy of the determinant.
\end{rem}


\bibliographystyle{alphaurl}
\bibliography{free}

\end{document}